\documentclass[10pt,a4paper]{amsart}

\usepackage{amsmath,amsfonts,amssymb,amsthm}
\newtheorem{theorem}{Theorem}[section]

\theoremstyle{definition}

\theoremstyle{remark}
\newtheorem{remark}[theorem]{Remark}

\numberwithin{equation}{section}
\newcommand{\nat}{\mathbb{N}}                                                   %zbior liczb naturalnych%
\newcommand{\rzecz}{\mathbb{R}}                                                 %zbior liczb rzeczywistych%
\newcommand{\zesp}{\mathbb{C}}                                                  %zbior liczb zespolonych%
\DeclareMathOperator{\supp}{supp}                                               %nosnik%
\newcommand{\indyk}[1]{1\!\!\!\! 1_{#1}}                                        %indykator zdarzenia%
\DeclareMathOperator{\MojeRe}{Re}                                               %czesc rzeczywista%
\DeclareMathOperator{\MojeIm}{Im}                                               %czesc urojona%
\newcommand{\iloskal}[2]{\left<{#1}|{#2}\right>}                                %iloczyn skalarny%
\newcommand{\cgw}{\textrm{C}^\ast}

\newmuskip\pFqmuskip

\newcommand*\pFq[6][8]{%
	\begingroup % only local assignments
	\pFqmuskip=#1mu\relax
	\mathcode`\,=\string"8000
	\begingroup\lccode`\~=`\,
	\lowercase{\endgroup\let~}\pFqcomma
	{}_{#2}F_{#3}{\left[\genfrac..{0pt}{}{#4}{#5};#6\right]}%
	\endgroup
}
\newcommand{\pFqcomma}{\mskip\pFqmuskip}

\newcommand{\Lag}[3]{L_{#1}^{(#2)}\left(#3\right)}

\begin{document}

\title[]{Hypergroups and quantum Bessel processes of non-integer dimensions}

\author[W.Matysiak]{Wojciech Matysiak}
\address{Wydzia{\l} Matematyki i Nauk Informacyjnych\\
Politechnika Warszawska\\
ul. Ko\-szy\-ko\-wa 75\\
00-662 Warsaw, Poland} \email{matysiak@mini.pw.edu.pl}

%\author[K.Szpojankowski]{Kamil Szpojankowski}
%\address{Wydzia{\l} Matematyki i Nauk Informacyjnych\\
%Politechnika Warszawska\\
%ul. Ko\-szy\-ko\-wa 75\\
%00-662 Warsaw, Poland} \email{K.Szpojankowski@mini.pw.edu.pl}

\thanks{Partially supported by NCN grant 2012/05/B/ST1/00554.}

\subjclass[2010]{Primary: 60J25; secondary 43A62.}

\keywords{Quantum Bessel process, Bessel process, hypergroup, Gelfand pair.}

\date{\today}

\begin{abstract}
It is demonstrated how to use certain family of commutative hypergroups to provide a universal construction of Biane's quantum Bessel processes of all dimensions not smaller than 1. The classical Bessel processes $\textrm{BES}(\delta)$ are analogously constructed with the aid of the Bessel-Kingman hypergroups for all, not necessarily integer, dimensions $\delta\ge1$.
\end{abstract}

\maketitle

\section{Introduction}\label{s:intro}

A Bessel process with index $\mu$ is a real-valued diffusion with infinitesimal generator
\[
\frac{1}{2}\frac{\textrm{d}^2}{\textrm{d}x^2}+\frac{2\mu+1}{2x}\frac{\textrm{d}}{\textrm{d}x}.
\]
The dimension of the Bessel process with index $\mu$ is $\delta=2(\mu+1)$ and a common notation for the Bessel process with the dimension $\delta$ is $\textrm{BES}(\delta)$. It is well known that the notion of the Bessel process makes sense for any real number $\delta$ and that the Bessel processes of integer dimension $\delta\ge2$ are the radial parts of a Brownian motion in dimension $\delta$, that is they can be defined by taking the Euclidean
distance from the origin of the $\delta$-dimensional Brownian motion. It also is well known that the Bessel processes of negative indices $\mu$ differ in a number of ways from the Bessel processes which have $\mu\ge0$. %extensions of Bessel process \cite{GoringJaeschkeYor}

Philippe Biane \cite{biane1996} introduced a curious analogue of the Bessel process, which he called the \emph{quantum Bessel process}, as an appropriately understood radial part of a noncommutative Brownian motion. Despite the name and the origin, the quantum Bessel process is a classical Markov process living in a subset of the real plane. (Its transition probabilities are listed in Section \ref{s:qBp}, see \eqref{main1}-\eqref{main5}.) 

Biane's construction is based on a Gelfand pair associated with the Heisenberg group $H$, realized as $H=\zesp^d\times\rzecz$. Unsurprisingly, the integer parameter $d$ appears in the formulas for the transition probabilities of the quantum Bessel process. It turns out, however, that these formulas themselves make perfect sense and define a Markov process for $d$ from a larger set (in particular, for some non-integer $d$). Thus the situation resembles the one with the usual Bessel process: taking the radial part of the Brownian motion results in the processes with the parameters, which do not cover the whole set of admissible parameters. (In contrast with the classical case, there is no natural notion of dimension associated with the noncommutative Brownian motion appearing in the construction of the quantum Bessel process.)

The purpose of this note is to present a way to extend Biane's construction onto a larger range of the parameter $d$. We emphasize that our point is not to check whether the extended process is correctly defined (which boils down to an easy computation) but to provide a uniform construction, similar in spirit to Biane's, in which the quantum Bessel processes with $d\in\nat$ arise as special cases.

The key idea of the extension is to replace the Heisenberg group with a hypergroup. A hypergroup is a locally compact space with a convolution product mapping each pair of points to a probability measure with compact support. The space is required to admit an involution that acts like the inverse operation in a group. The vehicle for our extension will be a concrete family of hypergroups, called the Laguerre hypergroups, which generalizes in a way the Gelfand pair associated with the Heisenberg group Biane's construction is based on. 

In addition to the analysis of the quantum Bessel process, we offer an analogous viewpoint on the classical Bessel process. Namely, we first present a construction of the $d$-dimensional $\textrm{BES}(d)$ Bessel process with $d\in\nat$, based on a Gelfand pair associated with the additive group $\rzecz^d$ and the special orthogonal group $\textrm{SO}(\rzecz^d)$. We do not claim novelty for this construction; as in the quantum Bessel process case, what seems to be new is an extension of the construction which covers the cases with non-integer dimensions $d$, % Here the vehicle for the extension is the Bessel-Kingman hypergroup.
the vehicle for which is the Bessel-Kingman hypergroup.

The rest of this note is structured as follows. Section \ref{s:qBp} contains a concise summary of Biane's construction of the quantum Bessel process; in Section \ref{s:bessel} we recall an analogous construction of the usual Bessel process (of integer dimension). In Section \ref{s:hypergroups} we provide basic definitions and terminology as well as some fundamental facts on commutative hypergroups and their harmonic analysis. The Bessel-Kingman hypergroups are presented in Section \ref{s:BesselKingman}, along with their application to an extension of the construction from Section \ref{s:bessel} onto non-integer dimensions. The extension of the quantum Bessel process is carried out in Section \ref{s:qbesselhyp}, after introducing the Laguerre hypergroups in Section \ref{s:Laguerre}. We conclude with Section \ref{s:conclremarks}, where we collect several miscellaneous remarks.

\subsection{Notation}
For a locally compact Hausdorff space $X$, let $M_b(X)$ denote the Banach space of bounded regular (complex) Borel measures with the total variation norm and $M^1(X)\subset M_b(X)$ the set of all probability measures. The set of bounded continuous functions on $X$ is $C_b(X)$. Point masses will be denoted by $\delta_x$. Lastly, $\mathbb{Z}_+=\nat\cup\{0\}=\{0,1,2,\ldots\}$.

\section{Biane's quantum Bessel process}\label{s:qBp}

In this section we briefly describe Biane's construction \cite{biane1996} of the quantum Bessel process.

Consider the $(2d+1)$-dimensional Heisenberg group realized as the set $H=\zesp^{d}\times\rzecz$ with the group law
$
(z,w)(z^\prime,w^\prime)=(z+z^\prime,w+w^\prime+\MojeIm\iloskal{z^\prime}{z}),
$
in which the inner product $\iloskal{\cdot}{\cdot}$ comes from $\zesp^d$ and $(z,w), (z^\prime,w^\prime)$ belong to $H$. It was observed in \cite{biane1996} that $\exp(t\psi)$ is positive definite on $H$ for $t>0$ when $\psi:H\to\zesp$ is defined as
\begin{equation*}
\psi(z,w)=-iw-\frac{|z|^2}{2},
%|z|^2/{2},
\end{equation*}
so $Q_t:L^1(H)\to L^1(H)$ given for all positive $t$ as
$
Q_t f(z,w)=\exp[t\psi(z,w)]f(z,w),
$
determines uniquely a semigroup of (completely) positive contractions on the group $C^\ast$-algebra $C^\ast(H)$ of $H$ (see \cite[1.2.1 Proposition]{biane1996}). Though $C^\ast(H)$ is noncommutative (as $H$ is nonabelian), and $(Q_t)_t$ might be viewed as the semigroup of noncommutative Brownian motion on the dual of $H$ (see \cite{biane1996}), $(Q_t)_t$ gives rise to a classical Markov semigroup in the following way. The unitary group $U(d)$ acts on the Heisenberg group $H$ by automorphisms
\begin{equation}\label{U_action}
u.(z,w)=(uz,w),\quad u\in U(d), (z,w)\in H.
\end{equation}
It is known 
%It has been known since the papers \cite{hularicci} by Hulanicki and Ricci and \cite{koranyi1982} by Kor\'{a}nyi 
that the convolution algebra $L^1_\textrm{rad}(H)$, consisting of the functions from $L^1(H)$ which are ``radial'', that is invariant under action \eqref{U_action}, is commutative. This can equivalently be stated: $(U(d)\ltimes H,U(d))$ is a Gelfand pair. For a friendly introduction to Gelfand pairs and other notions appearing in this note in context of Gelfand pairs, the reader is directed to look at \cite{vandijk}.

Completing $L^1_\textrm{rad}(H)$ to the (commutative) $\cgw$-algebra $\cgw_\textrm{rad}(H)$ of radial elements of $\cgw(H)$, one can use the Gelfand-Naimark theorem to observe that the restriction of $(Q_t)_t$ to $\cgw_\textrm{rad}(H)$ generates a Markov semigroup on the spectrum of $\cgw_\textrm{rad}(H)$. The Gelfand spectrum of $\cgw_\textrm{rad}(H)$ can be identified with the set of spherical functions of the Gelfand pair $(U(d)\ltimes H,U(d))$, which are either of the form $\varphi^{(d)}_{\tau,m}=\varphi^{(d)}_{\tau,m}(z,w)$ with $\tau\ne0$ and $m\in\mathbb{Z}_+$, where
\begin{equation}\label{SphFun1stKind}
\varphi^{(d)}_{\tau,m}(z,w)=%\binom{m+d-1}{m}^{-1}
\frac{m!\Gamma(d)}{\Gamma(m+d)}
\exp\left[i\tau w-\frac{1}{2}|\tau||z|^2\right]\Lag{m}{d-1}{|\tau||z|^2},
\end{equation}
or of the form $\varphi^{(d)}_{0,\mu}=\varphi^{(d)}_{0,\mu}(z,w)$ with $\mu>0$, where
\begin{equation}\label{SphFun2ndKind}
\varphi^{(d)}_{0,\mu}(z,w)=j_{d-1}(\mu|z|).
\end{equation}
Here and further on, $\Lag{k}{a}{x}=[(a+1)_k/k!]\mathstrut_1 F_1(-k;a+1;x)$ is the Laguerre polynomial,
%\[
%\Lag{k}{a}{x}=\frac{(a+1)_k}{k!}\pFq{1}{1}{-k}{a+1}{x},
%\]
%(the Laguerre polynomials are orthogonal on $[0,\infty)$, with the weight function $x^a e^{-x}$), 
$j_\nu$ denotes the normalized spherical Bessel function
$
j_\nu(z)=\Gamma(\nu+1)(z/2)^{-\nu}J_\nu(z),
$
and $J_\nu$ denotes the usual Bessel function of the first kind.
%\[
%J_\nu(z)=\sum_{n=0}^\infty \frac{(-1)^n(z/2)^{\nu+2n}}{\Gamma(n+\nu+1)n!}.
%\]

Moreover, the Gelfand spectrum of $\cgw_\textrm{rad}(H)$ can also be homeomorphically embedded in $\rzecz^2$, thus producing a classical Markov process, the quantum Bessel process, on a subset $\rzecz^2$. 
%As the subset of the real plane homeomorphic to the Gelfand spectrum of $\cgw_\textrm{rad}(H)$ Biane (see \cite[Section 2.4]{biane1996}) chooses the set
%\begin{multline}\label{BianeHeisFan}
%\left\{(x,kx):x<0, k\in\mathbb{Z}_+\setminus\{0,1,\ldots,d-1\}\right\}\cup\left\{(0,y):y\ge0\right\}\\
%\cup\left\{(x,kx):x>0, k\in\mathbb{Z}_+\right\},
%\end{multline}
%and then finds \cite[3.3.1 Proposition]{biane1996} some explicit formulas for the transition probabilities of the quantum Bessel process. 

As it was argued in \cite{biane1996} (without a proof), the semigroup of the quantum Bessel process can be identified as the unique semigroup $(q_t)_t$ satisfying
\begin{equation}\label{rozpiskawsfer}
\exp\left[t\psi(g)\right]\phi(g^{-1})=\int \phi^\prime(g^{-1})q_t(\phi,\textrm{d}\phi^\prime)
\end{equation}
for all $t>0$ and $g=(z,w)\in H$ and $\phi$ in the set of all positive definite spherical functions for $(U(d)\ltimes H,U(d))$ (the integral is over this set too); clearly, the identification of $(q_t)_t$ via \eqref{rozpiskawsfer} requires providing a homeomorphism between the set of spherical functions and a subset of $\rzecz^2$. To maintain coherence with our results from Section \ref{s:qbesselhyp}, which rely on Rentzsch's parametrization 
% \eqref{LagSemiDual} 
of the dual of the Laguerre semigroup (see Section \ref{s:Laguerre}), we will present a slightly modified version of Biane's process introduced in \cite{biane1996}. The only difference between Biane's process from \cite{biane1996} and the presented version is that the latter was computed using the set
\begin{equation}\label{myHeisFan}
\left\{(0,\mu^2)\in\rzecz^2:\mu\ge0\right\}\cup\bigcup_{k\in\mathbb{Z}_+}\left\{(\tau,k|\tau|):\tau\in\rzecz\setminus\{0\}\right\},
\end{equation}
and the mapping $(\tau,k)\mapsto\varphi^{(d)}_{\tau,k}$ for $\tau\ne0$, $(0,\mu^2)\mapsto\varphi^{(d)}_{0,2\mu}$ for $\mu\ge0$ (the fact that this is a homeomorphism between \eqref{myHeisFan} and the set of spherical functions of $(U(d)\ltimes H,U(d))$ follows directly from \cite[Lemma 6.1]{Rentzsch1999} - see the end of Section \ref{s:Laguerre} of this note), 
instead of the subset of $\rzecz^2$ and the homeomorphism chosen by Biane (cf. \cite[Section 2.4]{biane1996}). %As \eqref{myHeisFan} and Biane's subset are obviously homeomorphic, 
The difference on semigroups level is hardly significant and boils down to a shift in a parameter. 

The dynamics of the slightly modified quantum Bessel process is described by the following transition semigroup $(q_t)_{t>0}$ on the set \eqref{myHeisFan} (cf. \cite[3.3.1 Proposition]{biane1996}):
	\begin{enumerate}
		
		\item If $\mathbf{x}=(s,k|s|)$, $s<0$, and $u=s+t<0$, then 
		\begin{equation}\label{main1}
		q_t(\mathbf{x},\textrm{d}\mathbf{y})=\sum_{l=k}^\infty\frac{\Gamma(d+l)}{\Gamma(d+k)(l-k)!}\left(\frac{u}{s}\right)^{d+k} \left(1-\frac{u}{s}\right)^{l-k} \delta_{(u,-lu)}(\textrm{d}\mathbf{y}).
		\end{equation}
		
		\item If $\mathbf{x}=(s,k|s|)$ with $s<0$ and $t=-s$, then
		\begin{equation}\label{main2}
		q_t(\mathbf{x},\textrm{d}\mathbf{y})=\frac{1}{\Gamma(d+k)}\exp\left(-\frac{y_1}{t}\right)\left(\frac{y_1}{t}\right)^{d+k-1}\frac{1}{t}
		\left(\delta_{0}\otimes\textrm{Leb}\right)(\textrm{d}\mathbf{y})
		\end{equation}
		for $\mathbf{y}=(y_0,y_1)$ from the set \eqref{myHeisFan} ($\textrm{Leb}$ denotes the one-dimensional Le\-bes\-gue me\-asure).
		
		\item If $\mathbf{x}=(s,k|s|)$, $s<0$, and $u=s+t>0$, then 
		\begin{equation}\label{main3}
		q_t(\mathbf{x},\textrm{d}\mathbf{y})=\sum_{l=0}^\infty \frac{\Gamma(d+k+l)}{\Gamma(d+k)l!}\left(\frac{u}{t}\right)^{d+k}\left(-\frac{s}{t}\right)^l \delta_{(u,lu)}(\textrm{d}\mathbf{y}).
		\end{equation}
		
		\item If $\mathbf{x}=(0,y_1)$ with $y_1\ge0$, then
		\begin{equation}\label{main4}
		q_t(\mathbf{x},\textrm{d}\mathbf{y})=\sum_{l=0}^\infty \frac{1}{l!}\left(\frac{y_1}{t}\right)^l\exp\left(-\frac{y_1}{t}\right)\delta_{(t,lt)}(\textrm{d}\mathbf{y}).
		\end{equation}
		
		\item If $\mathbf{x}=(s,k|s|)$ with $s>0$, and $u=s+t$, then
		\begin{equation}\label{main5}
		q_t(\mathbf{x},\textrm{d}\mathbf{y})=\sum_{l=0}^k \binom{k}{l}\left(\frac{s}{u}\right)^l\left(1-\frac{s}{u}\right)^{k-l}\delta_{(u,lu)}(\textrm{d}\mathbf{y}).
		\end{equation}
		
	\end{enumerate}
	
It is not difficult to observe that equations \eqref{main1}-\eqref{main5} define the semigroup of a Markov process when $d\in(0,\infty)$; we will denote the Markov process with the semigroup $(q_t)_t$ given by these equations, and living on \eqref{myHeisFan}, by $\textrm{QBES}(d)$, and we will call it \emph{the quantum Bessel process of dimension $d$}. 

\begin{remark}
The quantum Bessel process can alternatively be constructed from a birth and death process (the Yule process), with no reference to the noncommutative Brownian motion, Heisenberg group or Gelfand pairs.
\end{remark}

\begin{remark}
Since the first coordinate of the quantum Bessel process always follows a uniform motion to the right, it makes sense to denote the whole process as $(t,X_t)_t$. It can be shown in a way similar to Theorem 4.1 in \cite{matysiakswieca1} that the second coordinate $(X_t)_t$ after a slight rescaling becomes a particular case of quadratic harness (for the definition, see \cite{BMW1}), called classical bi-Poisson process, which was first introduced in \cite{brycwesolo3} and presented as a member of a larger class of bi-Poisson processes in \cite{BMW2}. 
\end{remark}
	
%\begin{remark}\label{rem:1}
%As it was argued in \cite{biane1996} (without a proof), the semigroup $(q_t)_t$ can be identified as the unique semigroup $(q_t)_t$ satisfying
%\begin{equation}\label{rozpiskawsfer}
%\exp\left[t\psi(g)\right]\phi(g^{-1})=\int \phi^\prime(g^{-1})q_t(\phi,\textrm{d}\phi^\prime)
%\end{equation}
%for all $t>0$ and $g=(z,w)\in H$ and $\phi$ in the set of all positive definite spherical functions for $(U(d)\ltimes H,U(d))$ (the integral is over %this set too). 
%It was remarked in \cite{matysiakswieca1} that a proof of this fact can be carried out with the use of the Fourier transform on $L^1(U(d)\ltimes H,U(d))$ and the Bochner-Godement theorem for Gelfand pairs. %We will skip it; instead, 
%Due to our need to work with hypergroups, we shall prove an analogous statement, specifically tailored for commutative hypergroups (see Lemma %\ref{lem:qt}).
%Due to our need to work with hypergroups, we shall be using an analogous argumentation, specifically tailored for commutative hypergroups (see the proofs of theorems \ref{thm:BessKing} and \ref{thm:main}).
%\end{remark}

\section{A special view on Bessel processes}\label{s:bessel}

The name ``quantum Bessel process'', coined by Biane, is easily justifiable as the process described by \eqref{main1}-\eqref{main5} comes from the non-commutative (quantum) Brownian motion and it shares “the radiality” property with the ordinary Bessel process. To deepen the analogy between both kinds of Bessel processes, let us consider the following construction of the classical Bessel process (its detailed sketch can be found for example in \cite{biane1998}, but it generally is based on some well-known motives).

Let $G$ denote the abelian additive group $\rzecz^d$ and $|\cdot|$ the Euclidean norm on $\rzecz^d$. If $\psi:G\to[0,\infty)$, $\psi(x)=-|x|^2/2$ then it is well known that $G\ni x\mapsto \exp[t\psi(x)]$ is positive definite for all $t>0$, so it is easy to deduce that $Q_t f=\exp(t\psi)f$ defines a semigroup of positive contractions on the group algebra $L^1(G)$ of $G$. The semigroup can be extended in a natural way to the semigroup of positive contractions on the group $C^\ast$-algebra $C^\ast(G)$. %(see \cite[1.2.1 Proposition]{biane1996} once more). 
The dual group $\widehat{G}$ is isomorphic to $\rzecz^d$. This, along with commutativity of $C^\ast(G)$ (which follows from the fact that $G$ is abelian) and the Gelfand-Naimark theorem, implies that $(Q_t)_t$ is the convolution semigroup that convolves functions with the measures with the Fourier transforms $\exp(-t|x|^2/2)$. In other words, $(Q_t)_t$ is the semigroup of Brownian motion on $\rzecz^d$.

%Since a Euclidean motion of $\rzecz^d$ is any transformation of $\rzecz^d$ that can be written as the composition of a rotation and a translation, the group of Euclidean motions of $\rzecz^d$ is the semidirect product $\textrm{SO}(\rzecz^d)\ltimes\rzecz^d$. 
A well-known fact is that $(\textrm{SO}(\rzecz^d)\ltimes\rzecz^d,\textrm{SO}(\rzecz^d))$ is a Gelfand pair; equivalently the convolution subalgebra $L^1_\textrm{rad}(\rzecz^d)$ of $\textrm{SO}(\rzecz^d)$-invariant (radial) functions from $L^1(\rzecz^d)$ is commutative. Another well-known fact is that the spherical functions of positive type for this Gelfand pair are given by
\begin{equation}\label{SphFunEucMot}
\eta_u^{(d)}(x)=j_{d/2-1}(u|x|),\ u>0
\end{equation}
for $x\in\rzecz^d$; hence the Gelfand spectrum of $(\textrm{SO}(\rzecz^d)\ltimes\rzecz^d,\textrm{SO}(\rzecz^d))$ is isomorphic to $(0,\infty)$. (Observe that $\eta_u^{(1)}(x)=\cos(ux)$ because $J_{-1/2}(z)=\sqrt{2/(\pi z)}\cos z$).

In order to find the restriction $(p_t)_t$ of the semigroup $(Q_t)_t$ to $C^\ast_\textrm{rad}(\rzecz^d)$ one can invoke the arguments based on \eqref{rozpiskawsfer}, from which it follows that $(p_t)_t$ can be identified as the unique semigroup $(p_t)_t$ satisfying
\begin{equation}\label{rozpiska1}
\exp[-t|x|^2/2]\eta_u^{(d)}(-x)=\int_0^\infty\eta_v^{(d)}(-x)p_t(u,\textrm{d}v)
\end{equation}
for all $u>0$ and $x\in\rzecz^d$. The identification relies on the formula
%\begin{equation}\label{gradryz}
%\int_0^\infty v\exp\left(-\alpha v^2\right) I_\nu(\beta v) J_\nu(\gamma v) \textrm{d}v=\frac{1}{2\alpha}
%\exp\left(\frac{\beta^2-\gamma^2}{4\alpha}\right) J_\nu\left(\frac{\beta\gamma}{2\alpha}\right),
%\end{equation}
\begin{equation}\label{gradryz}
\int_0^\infty e^{-\alpha v^2} j_\nu(i\beta v)j_\nu(\gamma v)\frac{v^{2\nu+1}}{2^\nu\Gamma(\nu+1)}\textrm{d}v=
\frac{1}{(2\alpha)^{\nu+1}}\exp\left(\frac{\beta^2-\gamma^2}{4\alpha}\right)j_\nu\left(\frac{\beta\gamma}{2\alpha}\right),
\end{equation}
valid when $\MojeRe \nu>-1$ and $\MojeRe \alpha>0$ (see %\cite[formula 6.633, (4)]{GradshteynRyzhik2007} 
\cite[p. 395]{WatsonTreatiseBessel}). %; $I_\nu$ is the modified Bessel function of the first kind of order $\nu$, $I_\nu(z)=\exp(-i\pi\nu/2)J_\nu(iz)$.
%\[
%I_\nu(z)=\sum_{n=0}^\infty \frac{(z/2)^{\nu+2n}}{\Gamma(n+\nu+1)n!}.
%\]
With $\nu=d/2-1$, $\alpha=(2t)^{-1}$, $\gamma=|x|$ and $\beta=u/t$ we easily recover the semigroup 
%\begin{equation}\label{semibessel}
%p_t(x,\textrm{d}y)=\frac{y^{d-1}}{(xy)^{d/2-1}t}I_{d/2-1}\left(\frac{xy}{t}\right)\exp\left(-\frac{x^2+y^2}{2t}\right)
%\indyk{(0,\infty)}{(y)}\textrm{d}y
%\end{equation}
\begin{equation}\label{semibessel}
p_t(x,\textrm{d}y)=\frac{2y^{d-1}}{(2t)^{d/2}\Gamma(d/2)}j_{d/2-1}\left(\frac{ixy}{t}\right)\exp\left(-\frac{x^2+y^2}{2t}\right)
\indyk{(0,\infty)}{(y)}\textrm{d}y
\end{equation}
of the classical Bessel process $\textrm{BES}(d)$ of dimension $d$.

It is a standard fact that \eqref{semibessel} still properly defines a Markovian semigroup if $d$ is no longer integer-valued but takes on arbitrary positive values. Moreover, the condition $d>0$ is equivalent with the assumption on $\nu$ under which \eqref{gradryz} holds true. Nevertheless, in spite of the fact that \eqref{rozpiska1} by means of \eqref{gradryz} immediately gives \eqref{semibessel} for any $d>0$, one cannot extend directly  the above construction of the Bessel process onto arbitrary positive dimensions, as the construction relies on the Euclidean motion group $\textrm{SO}(\rzecz^d)\ltimes\rzecz^d$, which unavoidably requires $d$ to be a positive integer.

One of the aims of this note is to demonstrate that an extension of the above construction for all $d\ge1$ is possible if one replaces the Euclidean motion group with a certain family of hypergroups.

\section{Some elements of harmonic analysis on commutative hypergroups}\label{s:hypergroups}

A \emph{hypergroup} $(X,\ast)$ is a locally compact Hausdorff space $X$ together with a bilinear and associative multiplication $\ast$ (called \emph{convolution}) on $M_b(X)$ such that:
\begin{enumerate}
\item the mapping $(\mu,\nu)\mapsto \mu\ast\nu$ is weak-$^\ast$ continuous;
\item for all $x,y\in X$, the convolution $\delta_x\ast\delta_y$ is a compactly supported probability measure on $X$;
\item the map $(x,y)\mapsto\supp(\delta_x\ast\delta_y)$ from $X\times X$ into the space of all nonempty compact subsets of $X$ is continuous with respect to the Michael topology (for a definition, consult \cite{BloomHeyer1995});
\item there exists a neutral element $e\in X$ such that $\delta_e\ast\delta_x=\delta_x\ast\delta_e=\delta_x$ for all $x\in X$;
\item there exists a continuous involutive automorphism $x\mapsto \bar{x}$ on $X$ such that $\delta_{\bar{x}}\ast\delta_{\bar{y}}=(\delta_y\ast\delta_x)^{-}$ (if $\mu\in M_b(X)$ then $\mu^-$ is defined by $\mu^-(A)=\mu(\bar{A})$ for all Borel sets $A\subset X$), and $x=\bar{y}$ if and only if $e\in\supp(\delta_x\ast\delta_y)$.
\end{enumerate}

%Notice that the density of the finitely supported measures in $M_b(X)$ implies that the convolution $\ast$ extends uniquely from the point masses onto the whole space $M_b(X)$.

%The above set of axioms was given by Jewett in \cite{Jewett1975} (instead of ``hypergroups'', he used the term ``convos''); Dunkl \cite{Dunkl1973} and Spector \cite{Spector1978} obtained some important results for doing harmonic analysis on the Jewett's convos. This is one of the reasons that a ``hypergroup'' is now often referred to as a ``DJS-hypergroup''.

A locally compact group is an obvious example of a hypergroup; convolution is defined by $\delta_x\ast\delta_y=\delta_{xy}$ and the involution by $\bar{x}=x^{-1}$ (the group inverse of $x$).

A hypergroup $(X,\ast)$ is \emph{commutative} if the convolution $\ast$ is commutative. We collect here some basic facts %needed for our purposes and 
concerning the harmonic analysis on commutative hypergroups - consult the monograph \cite{BloomHeyer1995} for a very thorough treatment. Up to the end of this section, $(X,\ast)$ will be a commutative hypergroup.

There exists a unique (up to normalization) Haar measure $\omega_X$ on $X$. This is a positive Radon measure such that for all $x\in X$
\[
\int f(y)\textrm{d}\omega_X(y)=\int f_x(y)\textrm{d}\omega_X(y),
\]
where
$
f_x(y):=f(x\ast y):=\int f\textrm{d}(\delta_x\ast\delta_y).
$

If $f$ and $g$ are some measurable functions on $X$, then one defines their convolution $f\ast g$ as
\[
f\ast g(x)=\int_X f(y) g(x\ast \bar{y})\textrm{d}\omega_X(y);
\]
the involution $f^\ast$ is defined as $f^\ast(x)=\overline{f(\bar{x})}$. These allow to consider $L^1(X,\omega_X)$ as a (commutative) Banach $\ast$-algebra.

Analogously to the dual of a locally compact abelian group, it makes sense to consider the dual space $X^\wedge$ of $X$
\[
X^\wedge=\left\{\chi\in C_b(X):\chi\not\equiv0,\ \chi(x\ast \bar{y})=\chi(x)\overline{\chi(y)}\ \forall\ x,y\in X\right\},
\]
which is a locally compact Hausdorff space with respect to the topology of uniform convergence on compact sets. The dual space $X^\wedge$, whose elements are %unsurprisingly 
called \emph{characters} of $X$, can be identified with the symmetric spectrum %$\Delta_s(L^1(X,\omega_X))$ 
of %the commutative convolution Banach algebra 
$L^1(X,\omega_X)$. 
%Recall that the \emph{symmetric spectrum} $\Delta_s(\mathcal{A})$ of a unital commutative Banach $\ast$-algebra $\mathcal{A}$ is the set of all nonzero homomorphisms $\phi:\mathcal{A}\to\zesp$ satisfying the $\ast$-condition: $\phi(a^\ast)=\overline{\phi(a)}$ for all $a\in\mathcal{A}$. 
The symmetric spectrum is a locally compact Hausdorff space with respect to the weak$^\ast$-topology.

\begin{remark}\label{dualNOTdual}
As it is well known, every locally compact abelian group $G$ has a dual locally compact abelian group $\hat{G}$ and the dual of $\hat{G}$ is again $G$. One of the striking contrasts between hypergroups and group measure algebras is that the duality may fail in the hypergroup case: there exist commutative hypergroups $X$ such that $X^\wedge$ is not a hypergroup under pointwise multiplication (a hypergroup $X$ is called \emph{strong} if $X^\wedge$ is a hypergroup).
\end{remark}

The fundamental tools for the harmonic analysis on hypergroups are hypergroup Fourier transforms, defined in the following way. The Fourier transform of $f\in L^1(X,\omega_X)$ is
\[
\widehat{f}(\chi)=\int_X\overline{\chi(x)}f(x)\textrm{d}\omega_X(x),\quad\chi\in X^\wedge;
\]
the inverse Fourier transform of $\sigma\in M_b(X^\wedge)$ is defined on $X$ by
\[
\check{\sigma}(x)=\int_{X^\wedge}\chi(x)\sigma(\textrm{d}\chi);
\]
in particular, the inverse Fourier transform is injective. For a fixed Haar measure $\omega_X$ on $X$ there exists a unique positive Radon measure $\pi_X$ on $X^\wedge$ (called the Plancherel measure of $(X,\ast)$) such that $f\mapsto \widehat{f}$ extends to an $L^2$-isometry between $L^2(X,\omega_X)$ and $L^2(X^\wedge,\pi_X)$.

A \emph{positive definite} function on $X$ is a complex-valued, measurable and locally bounded function $f$ on $X$ such that
\[
\sum_{i=1}^{n}\sum_{j=1}^{n} c_i \overline{c_j} f(x_i\ast \overline{x_j})\ge 0
\]
for any choice of complex numbers $c_i$ and elements $x_i$ of $X$. Clearly, every character on $X$ is positive definite. However, unlike the group case, the product of two characters may not be positive definite %, and the product of two positive definite functions does not need to be positive definite 
(compare with Remark \ref{dualNOTdual}); if it however is, %(which is necessary for $X$ to be a so called \emph{weak hypergroup}), 
then the product of bounded continuous positive definite functions is positive definite. 

It is easy to see that if $\sigma$ is a bounded positive Radon measure on $X^\wedge$ then $\check{\sigma}$ is bounded and positive definite on $X$. Conversely, Bochner's theorem for commutative hypergroups \cite[4.1.16 Theorem]{BloomHeyer1995} asserts for any bounded, continuous and positive definite function $f$ the existence of a unique bounded positive Radon measure $\sigma$ on $X^\wedge$ such that $f=\check{\sigma}$ (clearly, if $f(e)=1$ then $\sigma$ is a probability measure).

%Now we can state and prove the lemma, which was mentioned in Remark \ref{rem:1} and which will be the main tool for finding the semigroups of Markov processes in the next two sections. Below we assume that $\psi:X\to\zesp$ is a bounded continuous function such that $\psi(e)=1$ (recall that $(X,\ast)$ is a commutative hypergroup).
%
%\begin{lemma}\label{lem:qt}
%Set $Q_t:L^1(X,\omega_X)\to L^1(X,\omega_X)$ by $Q_t f(x)=\exp[t\psi(x)]f(x)$ for $x\in X$ and $t>0$.
%
%If $X\ni x\mapsto\exp[t\psi(x)]\overline{\chi(x)}$ is positive definite on $X$ for every character $\chi$ and for every $t>0$, then there exists a unique semigroup $(q_t)_t$ of probability measures on $X^\wedge$ such that for all $\chi\in X^\wedge$ and $t>0$
%\begin{equation}\label{imageFourier}
%\widehat{(Q_t f)}(\chi)=\int_{X^\wedge} f(\chi^\prime)q_t(\chi,\textrm{d}\chi^\prime).
%\end{equation}
%\end{lemma}
%
%
%\begin{corollary}
%If $(X,\ast)$ is a strong hypergroup and $X\ni x\mapsto\exp[t\psi(x)]$ is positive definite on $X$ for all $t>0$, then the assertion of Lemma \ref{lem:qt} holds.
%\end{corollary}

\section{Bessel-Kingman hypergroups and extensions of Bessel processes}\label{s:BesselKingman}

Now we are in a position to introduce the Bessel-Kingman hypergroups, which will serve as a replacement for the Euclidean motion groups in the construction of the classical Bessel processes with arbitrary positive dimensions. 

Let $K=\rzecz_+$. For $\alpha>1$ we define the convolution $\ast_\alpha$ on $M^b(K)$ by
\[
f(x\ast_\alpha{x^\prime})=
%\delta_x\ast_\alpha\delta_{x^\prime}(f)=
\frac{\Gamma\left(\alpha/2\right)}{\sqrt{\pi}\Gamma((\alpha-1)/2)}\int_0^\pi f\left(\sqrt{x^2+{x^\prime}^2-2xx^\prime\cos\theta}\right)\sin^{\alpha-2}\theta\textrm{d}\theta
\]
for $x,x^\prime\in K$ and $f\in C_b(K)$; we also put
$
f(x\ast_1{x^\prime})=\left(f(x+x^\prime)+f(|x-x^\prime|)\right)/2.
$

This gives rise to a continuum of very well-studied commutative hypergroups $K_\alpha=(K,\ast_\alpha)$ with $\alpha\ge1$, known as the Bessel-Kingman hypergroups. They were discovered by J.F.C. Kingman in \cite{Kingman1963} (he did not use the term ``hypergroup''), in context of his research on spherically symmetric random walks.

The neutral element in each of the Bessel-Kingman hypergroups is 0, and the involution is the identity. The Haar measure of $K_\alpha$ is $\textrm{d}\omega_{K_\alpha}(x)=x^{\alpha-1}\textrm{d}x$ and the dual %$K_\alpha^\wedge$ 
of $K_\alpha$ is 
$
K_\alpha^\wedge=\left\{\eta_u^{(\alpha)}:u\ge0\right\}
$
- recall \eqref{SphFunEucMot}. Therefore the associated hypergroup Fourier transform is given by a Hankel transform. As it can be derived immediately from the symmetry of $\eta_u^{(\alpha)}(x)$ in $u$ and $x$, the
Bessel-Kingman hypergroups are self-dual (meaning that the dual $K_\alpha^\wedge$ is isomorphic to $K_\alpha$), like the underlying groups 
$\rzecz^d$. It implies, in particular, that $K_\alpha^\wedge$ is a (dual) hypergroup (so the Bessel-Kingman hypergroups are strong), and that the product of bounded continuous positive definite functions on $K_\alpha$ is positive definite. 

Notice that the dual $K_\alpha^\wedge$ of $K_\alpha$ coincides with the Gelfand spectrum of the Gelfand pair 
$(\textrm{SO}(\rzecz^\alpha)\ltimes\rzecz^\alpha,\textrm{SO}(\rzecz^\alpha))$ when $\alpha\in\nat$. This observation is in fact a glimpse of a larger picture: in the special case when $\alpha$ is an integer, the hypergroups $K_\alpha$ can also be inherited from a group or a Gelfand pair, since $(M^b(K),\ast_\alpha)$ is isometrically isomorphic to the subalgebra of rotation invariant measures on $\rzecz^\alpha$. Thus $K_\alpha$ interpolate these in a sense.

Therefore they are perfectly suitable for extending the construction of the Bessel process $\textrm{BES}(\delta)$ described in Section \ref{s:bessel} onto all, not only integer, dimensions $\delta\ge1$. The details of the extension express some well-known facts in the language of the Bessel-Kingman hypergroups, and they go as follows. Let $\delta\ge1$. In parallel to Section \ref{s:bessel}, we first consider function $\psi(x)=-x^2/2$ on $K$ and define $Q_t:L^1(K_\delta,\omega_{K_\delta})\to L^1(K_\delta,\omega_{K_\delta})$ for $t>0$ by $Q_tf=\exp(t\psi)f$.  

\begin{theorem}\label{thm:BessKing}
There exists a unique semigroup $(q_t)_t$ of %probability measures 
Markov kernels on $K_\delta^\wedge$ such that for all $\chi\in K_\delta^\wedge$ and $t>0$
\begin{equation}\label{imageFourier}
\widehat{(Q_t f)}(\chi)=\int_{K_\delta^\wedge} \widehat{f}(\chi^\prime)q_t(\chi,\textrm{d}\chi^\prime).
\end{equation}
The semigroup $(q_t)_t$ is exactly the semigroup of the $\delta$-dimensional Bessel process $\textrm{BES}(\delta)$.
\end{theorem}

\begin{proof}
From the definition of the Fourier transform it follows that
\begin{equation}\label{1}
\widehat{(Q_t f)}(\chi)%=\int_K (Q_t f)(x)\overline{\chi(x)}\textrm{d}\omega_{K_\delta}(x)
=\int_K f(x)\exp\left[t\psi(x)\right]\overline{\chi(x)}\textrm{d}\omega_{K_\delta}(x).
\end{equation}

%It is not difficult to observe that for any $t>0$, $K\ni x\mapsto\exp[t\psi(x)]$ is positive definite on the Bessel-Kingman hypergroup $K_\delta$. %Indeed, from 
%\[
%\int_0^\infty J_\nu(at) t^{\nu+1}e^{-p^2t^2}\textrm{d}t=\frac{a^\nu}{(2p^2)^{\nu+1}}\exp\left(-\frac{a^2}{4p^2}\right),\quad \MojeRe\nu>-1
%\]
%(this is a special case of the Weber-Sonine formula; see, for instance, %\cite[(4.11.25), p.222]{AndrewsAskeyRoy} or 
%\cite[formula 4 in 6.631]{GradshteynRyzhik2007}), 
%it follows that $\exp(t\psi)$ is the inverse Fourier transform $\check{\sigma}_t$ of a Gaussian measure on $K_\delta$, which is given by the Rayleigh %distribution
%\[
%\sigma_t(\textrm{d}u)=\frac{2}{(2t)^{\delta/2}\Gamma(\delta/2)}u^{\delta-1}\exp\left(-\frac{u^2}{2t}\right)\indyk{(0,\infty)}{(u)}\textrm{d}u.
%\]
Since $\exp(t\psi)$ is the inverse Fourier transform $\check{\sigma}_t$ of a Gaussian measure on $K_\delta$, which is given by the Rayleigh distribution (see \cite[Example 7.3.18]{BloomHeyer1995}), $K\ni x\mapsto\exp[t\psi(x)]$ is positive definite for any $t>0$ on the Bessel-Kingman hypergroup $K_\delta$. Consequently, $x\mapsto\exp[t\psi(x)]\overline{\chi(x)}$ is positive definite (recall the self-duality of the Bessel-Kingman hypergroups), and Bochner's theorem implies that it is the inverse Fourier transform of a unique measure $q_t(\chi,\cdot)$ on $K_\delta^\wedge$, which means that
\begin{equation}\label{glowne1}
\exp\left[t\psi(x)\right]\overline{\chi(x)}=\int_{{K_\delta^\wedge}} \chi^\prime(x)q_t(\chi,\textrm{d}\chi^\prime).
\end{equation}
We complete the proof of the unique existence by inserting \eqref{glowne1} into \eqref{1}, changing the order of integration and observing that the characters of the Bessel-Kingman hypergroups are real-valued.
%\begin{equation*}
%\widehat{(Q_t f)}(\chi)=\int_{{K_\delta^\wedge}} \left(\int_{K_\delta} f(x)\chi^\prime(x)\textrm{d}\omega_{K_\delta}(x)\right)q_t(\chi,\textrm{d}\chi^\prime).
%\end{equation*}
%Since the characters of the Bessel-Kingman hypergroups are real-valued, the expression in the parentheses is exactly the Fourier transform of $f$ and the proof of existence is complete.

Clearly, the semigroup $(q_t)_t$ can be recovered from \eqref{glowne1}. As $K_\delta^\wedge\cong[0,\infty)$, one immediately notices that \eqref{glowne1} differs from \eqref{rozpiska1} only by use of $\delta\ge1$ instead of the integer $d$. Since there is no need to assume in \eqref{gradryz} that $2\nu+2\in\nat$, we readily identify $(q_t)_t$ as the semigroup of the classical Bessel process of dimension $\delta\ge1$.

\end{proof}

\begin{remark}\label{rem:sufficient}
	Clearly, \eqref{glowne1} holding for all $\chi\in K_\delta^\wedge$, $x\in K$ and $t>0$ is sufficient for the existence of the unique semigroup $(q_t)_t$ from \eqref{imageFourier}. Additionally, as the Fourier transform for the Bessel-Kingman hypergroup is a Hankel transform, one can interpret \eqref{glowne1} in terms of an inverse Hankel transform; similarly, \eqref{gradryz} has an interpretation in terms of an (inverse) Hankel transform.
\end{remark}

\section{Laguerre hypergroups and quantum Bessel process}\label{s:Laguerre}

%Now, following Rentzsch \cite[Section 6]{Rentzsch1999}, 
Now we shall describe Laguerre hypergroups (see e.g. \cite{Rentzsch1999} and references therein), which arise from the orbit hypergroup structure carried by the space of all $U(d)$-orbits in the Heisenberg group $H=\zesp^d\times\rzecz$.

Let $K=\rzecz_+\times\rzecz$. For $\alpha>0$ we define the convolution $\ast$ on $M^b(K)$ by
\begin{multline}\label{LagHypConv}
f((x,s)\ast_\alpha(x^\prime,s^\prime))\\
%\delta_{(x,s)}\ast\delta_{(x^\prime,s^\prime)}(f)\\
=\frac{\alpha}{\pi}\int_{0}^{2\pi}\int_0^1f\left(\sqrt{x^2+y^2+2xyr\cos\theta},s+t+xyr\sin\theta\right)r(1-r^2)^{\alpha-1}\textrm{d}r\textrm{d}\theta;
\end{multline}
for $\alpha=0$ we put
\begin{equation*}
f((x,s)\ast_\alpha(x^\prime,s^\prime))
%\delta_{(x,s)}\ast\delta_{(x^\prime,s^\prime)}(f)
=\frac{1}{2\pi}\int_0^{2\pi}f\left(\sqrt{x^2+y^2+2xy\cos\theta},s+t+xy\sin\theta\right)\textrm{d}\theta;
\end{equation*}
in both cases $(x,s),(x^\prime,s^\prime)\in K$ and $f\in C_b(K)$. 

It is a well known fact that $K_\alpha=(K,\ast_\alpha)$ is a commutative hypergroup for $\alpha\ge0$ -- the \emph{Laguerre hypergroup (of order $\alpha$)}. The involution is given by $(x,s)\mapsto(x,-s)$, and the neutral element is $(0,0)$.

One can look upon the Laguerre hypergroups as a natural extension of the Gelfand pairs $(U(d)\ltimes H,U(d))$ discussed in Section \ref{s:qBp}, onto the case of non-integer $d\ge1$. This viewpoint is justified in particular by the comparison of the structure of the dual $K_\alpha^\wedge$ of $K_\alpha$, with the structure of the Gelfand spectrum of $(U(d)\ltimes H,U(d))$ (see Section \ref{s:qBp}). Namely, there are two kinds of characters in the dual of $K_\alpha$. The \emph{characters of the first kind} are given by the functions $\varphi^{(\alpha)}_{\tau,m}=\varphi^{(\alpha)}_{\tau,m}(x,w)$ (recall \eqref{SphFun1stKind})
%\begin{equation*}
%\phi_{\tau,m}(x,w)=\frac{m!\Gamma(\alpha)}{\Gamma(m+\alpha)}\exp\left[i\tau w-\frac{1}{2}|\tau|x^2\right]\Lag{m}{\alpha-1}{|\tau|x^2}
%\end{equation*}
with $(x,w)\in K$, $\tau\ne0$ and $m\in\mathbb{Z}_+$. The \emph{characters of the second kind} are $\varphi^{(\alpha)}_{0,\mu}=\varphi^{(\alpha)}_{0,\mu}(x,w)$ (recall \eqref{SphFun2ndKind}), with $\mu\ge0$.
%\begin{equation*}
%\phi_{0,\mu}(x,w)=j_{\alpha-1}(\mu x)%=\Gamma(\alpha)J_{\alpha-1}(\mu x)\left(\frac{\mu x}{2}\right)^{1-\alpha}
%\end{equation*}
%(they in fact do not depend on $w$). 
Thus, modulo the fact that $\alpha$ is not integer, the characters from $K_\alpha^\wedge$ coincide with the bounded spherical functions of the Gelfand pair built on the Heisenberg motion group. 

We shall use Rentzsch's (see \cite[(6.1),(6.2)]{Rentzsch1999}) specification of the dual $K_\alpha^\wedge$ of $K_\alpha$. Abusing the notation a bit, we will consider $K_\alpha^\wedge$ as the set given in \eqref{myHeisFan},
%\begin{equation}\label{LagSemiDual}
%K_\alpha^\wedge=\left\{(0,\mu^2)\in\rzecz^2:\mu\ge0\right\}\cup\bigcup_{k=0}^\infty\left\{(\tau,k|\tau|):\tau\in\rzecz\setminus\{0\}\right\},
%\end{equation}
with the following identification of the characters with the points from $K_\alpha^\wedge$:
\begin{equation}\label{ident}
\chi_{(\tau,k|\tau|)}=\varphi^{(\alpha)}_{\tau,k}\textrm{ for }\tau\ne0, \textrm{ and }\ 
\chi_{(0,\mu^2)}=\varphi^{(\alpha)}_{0,2\mu}\textrm{ for }\mu\ge0
\end{equation}
(the same as in Section \ref{s:qBp}). As it was proved in \cite[Lemma 6.1]{Rentzsch1999}, \eqref{ident} defines an isomorphism between the dual $K_\alpha^\wedge$ and the subset of $\rzecz^2$ described by \eqref{myHeisFan}.

\section{Extension of quantum Bessel process}\label{s:qbesselhyp}

We are now in a position to demonstrate how the Laguerre hypergroups can be used to construct the quantum Bessel processes $\textrm{QBES}(\delta)$ with any $\delta\ge1$. To this end, we fix $\delta\ge1$ and let $K_{\delta-1}=(K,\ast_{\delta-1})$ be the Laguerre hypergroup of order $\delta-1$. As in Section \ref{s:qBp}, we define $\psi:K\to\rzecz$ by
$
\psi(x,w)=-iw-x^2/2
$
for $(x,w)\in K$,
and $Q_t:L^1(K_{\delta-1},\omega_{K_{\delta-1}})\to L^1(K_{\delta-1},\omega_{K_{\delta-1}})$ by $Q_t f=\exp(t\psi)f$ for $t\ge0$. 

\begin{theorem}\label{thm:main}
There exists a unique semigroup $(q_t)_t$ of %probability measures 
Markov kernels on the dual $K_{\delta-1}^\wedge$ of the Laguerre hypergroup, which is generated, as in \eqref{imageFourier}, by the image of the semigroup $(Q_t)_t$ under the Fourier transform.
	
Under %\eqref{LagSemiDual} and 
\eqref{ident}, the semigroup $(q_t)_t$ is exactly the semigroup of the quantum Bessel process $\textrm{QBES}(\delta)$.
\end{theorem}

\begin{proof}

We need to prove %unique 
existence of such $(q_t)_t$ that for all $f\in L^1(K_{\delta-1},\omega_{K_{\delta-1}})$ and $\mathbf{x}\in K_{\delta-1}^\wedge$ (recall %\eqref{LagSemiDual} and 
\eqref{ident})
\begin{equation}\label{1Dec1}
\widehat{(Q_t f)}(\chi_{\mathbf{x}})=\int_{K^\wedge_{\delta-1}} \widehat{f}(\chi_{\mathbf{y}})q_t(\mathbf{x},\textrm{d}\mathbf{y})
\end{equation}
 -- uniqueness of $(q_t)_t$ follows immediately from the injectivity of the inverse hypergroup Fourier transform. Therefore, modifying appropriately the proof of Theorem \ref{thm:BessKing}, one immediately gets that Theorem \ref{thm:main} will be proved once it is shown that $(q_t)_t$ satisfying
 \begin{equation}\label{glowne3}
 \exp\left[t\psi(x,w)\right]\chi_\mathbf{x}(x,-w)=\int_{K_{\delta-1}^\wedge}\chi_\mathbf{y}(x,-w){q_t}(\mathbf{x},\textrm{d}\mathbf{y})
 \end{equation}
is the semigroup of the quantum Bessel process $\textrm{QBES}(\delta)$, given by \eqref{main1}-\eqref{main5} with $d$ replaced by $\delta$.

%The Fourier transform of $Q_tf$ is
%\[
%\int_{K^\wedge_\delta} f(x,w) \exp[t\psi(x,w)] \chi_{\mathbf{x}}(x,-w)\textrm{d}\omega_{K_\delta}(x,w).
%\]
%If there exists a semigroup $(\widetilde{q_t})_t$ on $K_\delta^\wedge$ such that
%\begin{equation}\label{glowne2}
%\exp\left[t\psi(x,w)\right]\chi_\mathbf{x}(x,-w)=\int_{K_\delta^\wedge}\chi_\mathbf{y}(x,w)\widetilde{q_t}(\mathbf{x},\textrm{d}\mathbf{y})
%\end{equation}
%holds for all $t>0$, $(x,w)\in K$, and $\mathbf{x}\in K_\delta^\wedge$, then $\exp\left[t\psi\right]\overline{\chi_\mathbf{x}}$ is a positive definite function, and such $(\widetilde{q_t})_t$ is unique by Bochner's theorem. Set $q_t(\cdot,A)=\widetilde{q_t}(\cdot,-A)$ for $A\subset K_\delta^\wedge$. Obviously, unique existence of $(q_t)_t$ is equivalent to unique existence of $(\widetilde{q_t})_t$. Since $\chi_{\mathbf{y}}(x,w)=\overline{\chi_{-\mathbf{y}}(x,w)}=\chi_{-\mathbf{y}}(x,-w)$, \eqref{glowne2} holds if and only if
%%\begin{equation}\label{glowne3}
%%\exp\left[t\psi(x,w)\right]\chi_\mathbf{x}(x,-w)=\int_{K_\delta^\wedge}\chi_\mathbf{y}(x,-w){q_t}(\mathbf{x},\textrm{d}\mathbf{y})
%%\end{equation}
%holds. Moreover, modifying appropriately the proof of Theorem \ref{thm:BessKing}, one immediately gets that \eqref{glowne3} implies \eqref{1Dec1}. Thus Theorem \ref{thm:main} will be proved once it is shown that $(q_t)_t$ satisfying \eqref{glowne3} is the semigroup of the quantum Bessel process $\textrm{QBES}(\delta)$, given by \eqref{main1}-\eqref{main5} with $d$ replaced by $\delta$.

The rest of the proof consists of a careful analysis of \eqref{glowne3} in all the cases leading to \eqref{main1}-\eqref{main5} (with $\delta$ instead of $d$), and is based on the same general approach laid out in \cite{matysiakswieca1}. But, for completeness, we outline it here (see also Remark \ref{rem:altProof}). 

First we aim to prove \eqref{main1}. If $\mathbf{x}=(s,k|s|)$ with $s<0$ and $u=s+t<0$ then the LHS of \eqref{glowne3} equals 
$[k!/(\delta)_k]\exp(-iuw+{ux^2}/{2})\exp(-tx^2)\Lag{k}{\delta-1}{-sx^2}$.
%\[
%\frac{k!}{(\delta)_k}\exp\left(-iuw+\frac{ux^2}{2}\right)\exp(-tx^2)\Lag{k}{\delta-1}{-sx^2}.
%\]
Therefore \eqref{main1} follows from the identity
	\[
	\sum_{i=0}^\infty\frac{(i+j)!}{i!j!}\Lag{i+j}{\alpha}{v}\tau^i=(1-\tau)^{-\alpha-1-j}\exp\left(-\frac{v\tau}{1-\tau}\right)\Lag{j}{\alpha}{\frac{v}{1-\tau}}
	\]
	(one can find it, for example, in \cite[(4.6.34), p.104]{Ismail2009}), via the substitutions $\alpha=\delta-1$, $v=-ux^2$, $\tau=-t/s$,  and some straightforward manipulations.
	
	Now we will prove \eqref{main2}. From the identity
	\[
	k!\Lag{k}{\alpha}{u}=u^{-\alpha/2}\int_0^\infty e^{u-v} v^{k+\alpha/2} J_\alpha(2\sqrt{uv})\textrm{d}v
	\]
	(see \cite[(4.6.42), p.105]{Ismail2009}), it follows that the LHS of \eqref{glowne3}, which is in this case equal to
	$
	({k!}/{(\delta)_k})\exp({sx^2})\Lag{k}{\delta-1}{-sx^2},
	$
	can be written, using $\alpha=\delta-1$ and $u=-sx^2$ and a linear change of variable, as
	\[
	\int_0^\infty \Gamma(\delta)J_{\delta-1}(2\sqrt{r}x)\left(\sqrt{r}x\right)^{1-\delta}\cdot\frac{1}{\Gamma(\delta+k)}e^{r/s}\left(-\frac{r}{s}\right)^{k+\delta-1}\left(-\frac{1}{s}\right)
	\textrm{d}r.
	\]
	Due to our specification of the dual of the Laguerre hypergroup, this proves \eqref{main2}.

	Now let us turn to proving \eqref{main3}. To this end, observe that the formula \cite[(4.6.31), p.103]{Ismail2009}
	\[
	\sum_{l=0}^\infty \frac{(c)_l}{(\alpha+1)_l}\Lag{l}{\alpha}{v}\tau^l=(1-\tau)^{-c}\pFq{1}{1}{c}{\alpha+1}{-\frac{v\tau}{1-\tau}},
	\]
	together with the identity $\mathstrut_1 F_1(a;b;z)=\exp(z)\mathstrut_1 F_1(b-a;b;-z)$ \cite[(1.4.11), p.13]{Ismail2009},
%	\[
%	\pFq{1}{1}{a}{b}{z}=\exp(z)\pFq{1}{1}{b-a}{b}{-z},
%	\] 
	and the definition of the Laguerre polynomials, when specified for $c=\alpha+1+k$, yield
	\begin{equation}\label{BlackFriday1}
	\frac{k!}{(\alpha+1)_k}\exp\left(-\frac{v\tau}{1-\tau}\right)\Lag{k}{\alpha}{\frac{v\tau}{1-\tau}}=(1-\tau)^{\alpha+1+k}
	\sum_{l=0}^\infty \frac{(\alpha+1+k)_l}{(\alpha+1)_l}\Lag{l}{\alpha}{v}{\tau^l}.
	\end{equation}
	Since the LHS of \eqref{glowne3} for $\mathbf{x}=(s,|k|s)$, $s<0$, and $u=s+t>0$, is 
	\[
	\frac{k!\Gamma(\delta)}{\Gamma(k+\delta)}\exp\left(-iuw-\frac{ux^2}{2}\right)\exp(sx^2)\Lag{k}{\delta-1}{-sx^2},
	\]
	\eqref{main3} follows by plugging $\alpha=\delta-1$, $\tau=-s/t$ and $v=ux^2$ into \eqref{BlackFriday1}.
	
	If the assumptions of \eqref{main4} hold then the specification of the dual we use implies that the LHS of \eqref{glowne3} is
	$
	\exp\left[t\left(-iw-x^2/2\right)\right]\Gamma(\delta)J_{\delta-1}(2x\sqrt{y_1})(x\sqrt{y_1})^{1-\delta}.
	$
	Since the RHS of the identity
	\begin{equation}\label{BlackFriday2}
	\sum_{l=0}^{\infty}\frac{\Lag{l}{\alpha}{v}}{(\alpha+1)_l}\tau^l=\exp(\tau)\ \pFq{0}{1}{-}{\alpha+1}{-v\tau}
	\end{equation}
	(see e.g. \cite[(4.6.30), p.103]{Ismail2009}) is easily seen to be equal to $\exp(\tau)j_\alpha(2\sqrt{v\tau})$,
%	\[
%	\Gamma(\alpha+1)\exp(\tau)J_\alpha(2\sqrt{v\tau})(v\tau)^{-\alpha/2}=\exp(\tau)j_\alpha(2\sqrt{v\tau}),
%	\]
	putting $\alpha=\delta-1$, $v=tx^2$ and $\tau=y_1/t$ into \eqref{BlackFriday2}, one immediately arrives at \eqref{main4}.

	We proceed to justify \eqref{main5}. Since in this case the LHS of \eqref{glowne3} is equal to 
	\[
	\frac{k!}{(\delta)_k}\exp\left(-iuw-\frac{ux^2}{2}\right)\Lag{k}{\delta-1}{sx^2},
	\]
	it suffices to apply the identity
	\[
	\Lag{k}{\alpha}{cv}=(\alpha+1)_k\sum_{l=0}^k\frac{c^l(1-c)^{k-l}}{(k-l)!(\alpha+1)_l}\Lag{l}{\alpha}{v}
	\]
	(which can be found in \cite[(4.6.23), p.103]{Ismail2009}) with $\alpha=\delta-1$, $v=ux^2$ and $c=s/u$, to obtain \eqref{main5} readily.
	
	We have therefore identified $(q_t)_t$ from \eqref{glowne3} as the semigroup of $\textrm{QBES}(\delta)$. The proof is complete.
\end{proof}

\section{Concluding remarks}\label{s:conclremarks}

\begin{remark}
Although it seems 
possible to carry the extensions from sections \ref{s:BesselKingman} and \ref{s:qbesselhyp} without a direct appeal to hypergroup theory, well-developed harmonic analysis on commutative hypergroups provides all tools needed to accomplish our tasks, so we decided to stick to this convenient setup.

Let us note, however, that another approach for the extension of the quantum Bessel process might be based on a commutative algebra constructed by Stempak \cite{Stempak1988} by his careful analysis of a generalization of the sublaplacian on the Heisenberg group. In fact, results from \cite{Stempak1988} immediately lead to the Laguerre hypergroup structure.
\end{remark}

\begin{remark}\label{rem:altProof}
It is worth mentioning that, although Biane suggests (see \cite[p.95]{biane1996}) a possibility of deriving formulas \eqref{main1}-\eqref{main5} from some properties of the Laguerre polynomials and Bessel functions (presumably the ones we used in the proof of Theorem \ref{thm:main}), his proof of \eqref{main1}-\eqref{main5} (with $d\in\nat$) is completely different (\emph{``closer to the spirit of quantum probability''}, as Biane put it in \cite{biane1996}) from our proof of Theorem \ref{thm:main}, and does not seem to carry over directly to the case $d\in[1,\infty)$. 
\end{remark}

\begin{remark}\label{rem:BesselnotExtend}
In Section \ref{s:BesselKingman} we constructed the classical Bessel processes $\textrm{BES}(\delta)$ of all dimensions $\delta\ge1$. One can wonder whether it is possible to provide a similar construction of the Bessel processes $\textrm{BES}(\delta)$ with $\delta<1$. However, it needs to be noted that the Bessel-Kingman hypergroups $K_\delta$, as described in Section \ref{s:BesselKingman}, make sense for $\delta\ge1$ only. The reason for that is hidden (see the proof of Theorem 1 in \cite{Kingman1963}) in the range of the parameter $\nu$ in Gegenbauer's product formula 
\[
\frac{J_\nu(x)J_\nu(y)}{(xy)^\nu}=\frac{1}{2^\nu\Gamma(\nu+1/2)\sqrt{\pi}}
\int_0^\pi\frac{J_\nu(\sqrt{x^2+y^2-2xy\cos\theta})}{(x^2+y^2-2xy\cos\theta)^{\nu/2}}\sin^{2\nu}(\theta)\textrm{d}\theta,
\]
valid for $x,y>0$ and $\nu>-1/2$, which is equivalent (with $\nu=\delta/2-1$) to the associativity of the Bessel-Kingman convolution $\ast_\delta$ for $\delta>1$; when $\nu=-1/2$, Gegenbauer's formula degenerates to the well-known trigonometric formula for the product of cosines.
%\[
%\cos x\cos y=\frac{1}{2}\left(\cos(x-y)+\cos(x+y)\right).
%\]
\end{remark}

\begin{remark}
A situation similar, though not identical, to the one described in Remark \ref{rem:BesselnotExtend} occurs in the case of quantum Bessel processes and Laguerre hypergroups. The construction from Section \ref{s:qbesselhyp} is valid for all $\delta\ge1$, while it is not difficult to check that formulas \eqref{main1}-\eqref{main5} correctly define $\textrm{QBES}(\delta)$ processes for all $\delta>0$. The question whether one can construct these processes in a similar manner to what was demonstrated in Section \ref{s:qbesselhyp} is open. 

Let us notice however that the Laguerre hypergroups $K_{\delta-1}$ do not seem to be the right tool for the construction of $\textrm{QBES}(\delta)$ processes with $\delta<1$ as they make sense for $\delta\ge1$ only (the weight from the right hand side of \eqref{LagHypConv} is integrable only if $\alpha>0$). The range of the parameter $\nu$ in Watson's formula %\cite{Watson1939}
\begin{multline*}
\Lag{k}{\nu}{x^2}\Lag{k}{\nu}{y^2}=\frac{2^{\nu-1/2}\Gamma(\nu+k+1)}{\Gamma(k+1)\sqrt{\pi}}\\
\times
\int_0^\pi e^{xy\cos\theta}
\frac{J_{\nu-1/2}(xy\sin\theta)}{(xy\sin\theta)^{\nu-1/2}}
\Lag{k}{\nu}{x^2+y^2-2xy\cos\theta}\sin^{2\nu}\theta\textrm{d}\theta
\end{multline*}
(valid for $\MojeRe\nu>-1/2$, $x,y\in\rzecz$ and $k\in\nat$), on which the definition of the Laguerre hypergroup is based (see \cite{Koornwinder1977} or \cite{Stempak1988}), is slightly wider: setting $\nu=\delta-1$, as is needed to define the Laguerre hypergroup $K_{\delta-1}$, one gets $\delta>1/2$. It seems worth recalling that Watson's formula holds for $\nu\in(-1/2,0)$, or equivalently $\delta\in(1/2,1)$, by analytic continuation.

\end{remark}

\begin{remark}
Some multidimensional versions of Biane's quantum Bessel processes were obtained in \cite{matysiakswieca1} and \cite{matysiakswieca2}.  Thanks to the work of Voit \cite{Voit2011}, who introduced certain multidimensional hypergroups related to the Heisenberg groups, it seems possible to generalize some of the multidimensional quantum Bessel processes in the spirit of this paper. Analogously, one may wonder whether the results of R\"{o}sler \cite{Rosler2007} can be applied to obtain some multidimensional versions of the Bessel processes (see also \cite{Voit2009}) -- this is an area of our future work.
\end{remark}

\subsection*{Acknowledgment} We thank Jacek Weso{\l}owski, Marcin {\'S}wieca and Kamil Szpojankowski for helpful discussions concerning this work, and Walter Bloom and Herbert Heyer for some bibliographical information. We gratefully acknowledge insightful remarks of Krzysztof Stempak regarding certain possibilities of the Laguerre hypergroup's extensions. It is a pleasure to thank the referee for a careful examination of the paper, and
for many insightful suggestions and corrections. 

\bibliographystyle{amsplain}
\bibliography{../../WMbib}

\end{document}